\documentclass[oneside]{article}

\usepackage{amssymb,amsfonts}
\usepackage{amsthm}
\usepackage{cite}
\usepackage{amsmath}
\usepackage{graphicx}
\usepackage{color}
\usepackage[active]{srcltx}
\usepackage{multirow}
\usepackage{hyperref}
\usepackage[yyyymmdd]{datetime}

\usepackage{caption}
\usepackage{subcaption}

\numberwithin{equation}{section}

\newtheorem{theorem}{Theorem}[section]
\newtheorem{lemma}[theorem]{Lemma}

\newtheorem{proposition}[theorem]{Proposition}
\newtheorem{definition}[theorem]{Definition}
\newtheorem*{theorem*}{Theorem}

\newtheorem{theorem**}{Maximum Principle}
\theoremstyle{remark}
\newtheorem{remark}[theorem]{\bf{Remark}}

\title{On Phragm\'en-Lindel\"of  principle  for Non-divergence Type Elliptic Equation and  Mixed Boundary conditions}
\author{Akif Ibraguimov\footnote{ Email: {akif.ibraguimov@ttu.edu}. 
{Partially supported by DMS NSF grant 1412796}}~ and
Alexander I. Nazarov
\footnote{
St.Petersburg Department of Steklov Institute, Fontanka 27, St.Petersburg, 191023, Russia, 
and St.Petersburg State University, 
Universitetskii pr. 28, St.Petersburg, 198504, Russia. E-mail: al.il.nazarov@gmail.com.
Supported by RFBR grant 15-01-07650.
}
\date{}
}

\begin{document}

\maketitle
\begin{abstract} 
Paper dedicated to qualitative study of the solution of the Zaremba type problem in Lipschitz domain with respect to the elliptic equation in non-divergent form. Main result is Landis type Growth Lemma in spherical layer for Mixed Boundary Value Problem in the class of "admissible domain". Based on the Growth Lemma Phragm\'en-Lindel\"of theorem is proved at junction point of Dirichlet boundary and boundary over which derivative in non-tangential direction is defined. 
\end{abstract}

\section{Introduction}

We consider  non-divergence elliptic operator
\begin{equation}\label{equ}
{\cal L}u := - \sum_{i,j=1}^{n}a_{i j}(x) D_iD_ju \qquad\mbox{in}\quad \Omega.
\end{equation}
Such operators arise in theory of stochastic processes and various applications.

In (\ref{equ}) $\Omega$ is a domain in $\mathbb{R}^n$, $n\geq 3$, and $D_i$ stands for the differentiation with respect to $x_i$.
We suppose that the boundary $\partial \Omega$ is split   $\partial \Omega=\Gamma_1\cup\{\zeta\}\cup \Gamma_2.$ Here $\Gamma_1$ is support of the Dirichlet condition, 
and $\Gamma_2$ is support of the oblique derivative condition:
\begin{equation*}\label{BC}
u(x)=\Phi(x) \ \ \mbox{on} \ \ \Gamma_1;\quad
\frac{\partial u}{\partial \ell}(x):=\lim_{\delta \to +0} \frac {u(x)-u(x-\delta \ell)}{\delta}=\Psi(x) \ \ \mbox{on} \ \ \Gamma_2,
\end {equation*}
where $\ell=\ell(x)$ is a measurable,  and uniformly non-tangential outward vector field on $\Gamma_2$. Without loss of generality we can suppose $|\ell|\equiv1$. 
We call $\Gamma_1$ Dirichlet boundary, and $\Gamma_2$ Neumann boundary.

At point $\zeta\in\overline{\Gamma_1}\cap\overline{\Gamma_2}$ function $u$ is not defined, and we investigate asymptotic properies of the solution at this point.

For divergence type equation in case of Dirichlet Data this type of theorem first was proved in very general case by Mazya in \cite{MazyaFL1}. Criteria for regularity 
for Zaremba problem first was obtained by Mazya in \cite{MazyaFL}.

Here we consider the case of non-divergence equation in bounded domain $\Omega$   where Neumann  $\Gamma_2$ is Lipschitz in a neighborhood of the point $\zeta$.
 
 In the case $\Gamma_2=\emptyset$ the similar question was discussed by E.M. Landis (see \cite{landis-book,landis-paper}) and sharpened by Yu.A. Alkhutov \cite{Alkhutov}.

We always assume that the matrix of leading coefficients $(a_{ij})$ is bounded, measurable and symmetric, and satisfies the uniform ellipticity 
condition:
\begin{equation*}\label{e1} 
\max_{|\xi|=1}\sup_{x\in \Omega} e(x,\xi)=:e_1<\infty,
\end{equation*}
where  $e$ is the ellipticity function (see \cite{landis-book}, \cite{Alkhutov})
\begin{equation*}\label{ ellipticity function}
e(x,\xi)=  \frac{\sum _{i=1}^{n} a_{ii}(x)}{\sum_{i,j=1}^{n}a_{i j}(x)\xi_i\xi_j}.
\end{equation*}
For simplicity we consider the operators without lower-order terms, a more general case can be easily managed.



The paper is organized as follows. 

In Sec.~\ref{sec:preliminary} we formulate some known results about non-divergence equations:  lemma  on non-tangential derivatives 
at point of maximum (minimum) on the boundary in the form of Nadirashvili \cite{nad-max-principle}, the Landis Growth Lemma in case $\Gamma_2 =\emptyset$, and Growth Lemma in Krylov's form(see \cite{krylovizvest}).

The Growth Lemma for elliptic and parabolic equations first was introduced by Landis in \cite{Landis2, Landis1}. Growth Lemma is a fundamental tool to study 
qualitative properties and regularity of solutions in bounded and unbounded domain. Recent review on Growth Lemma and its applications was published in 
\cite{safonov-1} (see also \cite{Aimar}).


In Sec.~\ref{sec:Lip-boundary} we prove strict Growth Lemma near Neumann boundary.

Sec.~\ref{sec:Spherical-Layer} glues two Growth Lemmas. This result was obtained under 
some admissibility constraint on the boundary $\Gamma_2$, which is an analog of isoperimetric condition.

In the last Sec.~\ref{sec:Dichotomy}, dichotomy theorem is proved for solutions of mixed boundary value problem to non-divergence elliptic equation.  

We use the following notation. $x=(x',x_n)=(x_1,\dots, x_{n-1}, x_n$ is a point in $\mathbb R^n$. $B(x,R)$ is the ball centered in $x$ with radius $R$.


\section{Preliminary Results}\label{sec:preliminary}

Here we recall some known results and prove  auxiliary lemmas for the sub- and supersolution of the equation ${\cal L}u=0$.   
We call function $u$ sub-elliptic (super-elliptic) if $u \in W^2_n(\Omega)\bigcap {\cal C}^{1}(\Omega\cup\Gamma_2)$, and  
 ${\cal L}u\le 0$ (respectively, ${\cal L}u\ge 0$).
 

We say that $\Gamma_2$ satisfies inner cone condition (see, e.g., \cite{nad-max-principle}) 
if there are $0<\varphi<\pi/2$ and $h>0$ such that for any $y\in \Gamma_2$ there exists a 
right cone $K(y)\subset \Omega$ with the apex at $y$, apex angle $\varphi$ and of the height $h$.     

\begin{figure}[ht]
\begin{center}
\advance\leftskip-3cm
\advance\rightskip-3cm
\includegraphics[scale=0.5]{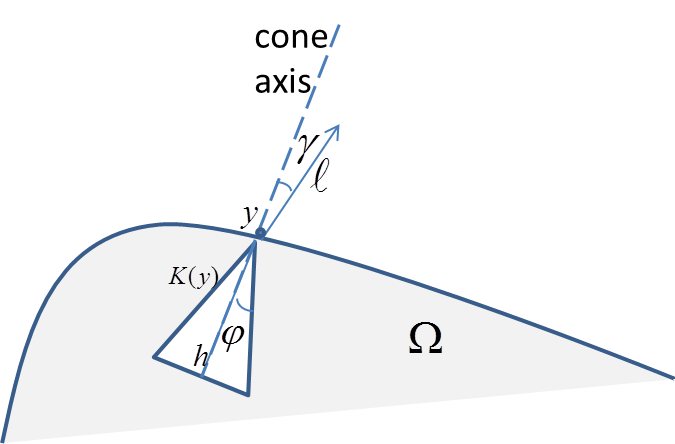}
\caption{Inner cone condition }.
\label{Cone}
\end{center}\end{figure}

In \cite{nad-max-principle} N. Nadirashvili obtained  fundamental generalization of Oleinik-Hopf lemma\footnote{In \cite{nad-max-principle} classical solutions
$u\in {\cal C}^2(\Omega)\cap {\cal C}^1(\overline\Omega)$ are used but due to the Aleksandrov-Bakel'man maximum principle it is transferred to 
$u \in W^2_n(\Omega)\bigcap {\cal C}^{1}(\Omega\cup\Gamma_2)$.}, the so-called ``lemma on non-tangential derivative'':

\begin{lemma}\label{nad} Let $\Gamma_2$ satisfy inner cone condition. Let a non-constant function $u$ be super-elliptic (sub-elliptic) 
${\cal L} u \geq 0 \ ({\cal L}u\leq 0)$ in $\Omega$. Suppose that $y\in\Gamma_2$ and $u(y)\le u(x) \ (u(y)\ge u(x))$ for all $x \in \Gamma_2$. 
Then for any  neighborhood $S$  of $y$ on $\Gamma_2$ and for any $\varepsilon<\varphi$ there exists a point $\widetilde x \in S$ s.t.
\begin{equation*}
\frac{\partial u}{\partial \ell}(\widetilde x)<0 \qquad \Big(\frac{\partial u}{\partial \ell}(\widetilde x)>0\Big) 
\end{equation*}
for any outward direction $\ell$ s.t. the angle $\gamma$ between $\ell$ and the axis of $K(\widetilde x)$ is not greater then $\varphi-\varepsilon$.   
\end{lemma}

From standard maximum principle and Lemma \ref{nad} follows comparison theorem for mixed boundary value problem.

\begin{lemma}\label{comparison_theorem}
Let $\Omega$ be a bounded domain, $\partial \Omega =\Gamma_1\cup \Gamma_2$. Let $\Gamma_2$ satisfy inner cone condition. 
Suppose that vector field $\ell$ satisfies the same condition as in Lemma \ref{nad}.
Let functions $u$ and $v$ belong to $W^2_n(\Omega)\bigcap {\cal C}^{1}(\Omega\cup\Gamma_2)\cap {\cal C}(\overline \Omega)$.

Then, if ${\cal L} u \leq {\cal L} v  $ in $\Omega$, $u\leq v $ on $\Gamma_1$, and
$\frac{\partial u}{\partial \ell}\leq \frac{\partial v}{\partial \ell}$ on $\Gamma_2$ then $u\geq v $ in $\overline \Omega$.
\end{lemma}
%
%
%

\begin{definition}\label{strict_growth}
Let $\Omega$ be a domain, $\partial \Omega=\Gamma_1\cup \Gamma_2$. Define ``small ball'' $B(0,R)$ and ``big ball'' $B(0,aR)$, $a>1$ (see Fig. \ref{domainandballs}).

We call the function $w$ {\bf barrier} with respect to mixed boundary value problem in these two balls if it posses properties:
\begin{equation}\label{Lw<0}
 w \ \text{is sub-elliptic } \  ({\cal L}w\leq 0) \ \text{in the intersection} \ \Omega\cap B(0,aR);
\end{equation}
\begin{equation}\label{w<1}
w(x)\leq 1 \ \text{on} \ \Gamma_1\cap B(0,aR);
\end{equation} 
\begin{equation}\label{w_l<0}
\frac{\partial w}{\partial \ell}\leq 0 \ \text{on} \ \Gamma_2\cap B(0,aR);
\end{equation} 
\begin{equation}\label{w<0}
w\leq 0 \ \text{on} \ \overline \Omega\cap \partial B(0,aR);
\end{equation}
\begin{equation}\label{w>eta}
w(x)\ge \eta_0 \ \text{in the intesection} \
 B(0,R)\cap \Omega
\end{equation} 
for some constant $\eta_0$.  

\end{definition}

\begin{figure}[ht]
\begin{center}
\advance\leftskip-3cm
\advance\rightskip-3cm
\includegraphics[scale=0.5]{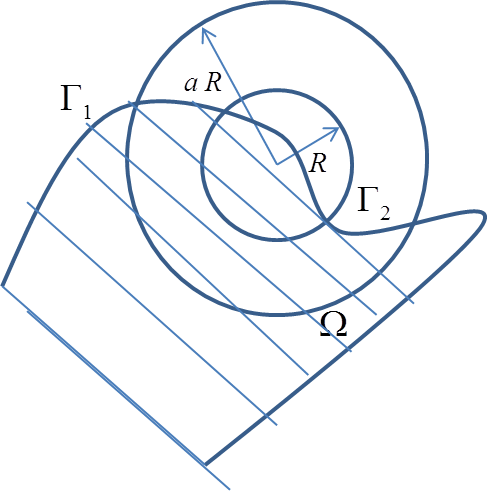}
\caption{Domain $G$ and two balls $B(0,R)$ and $B(0,aR),\ (a>1) \ $}
\label{domainandballs}
\end{center}\end{figure}

Now we are in the position to prove the following strict growth property for subsolutions of the mixed boundary value problem.

\begin{lemma}\label{krylov}
Let $\Omega$ be a domain, $\partial \Omega=\Gamma_1\cup \Gamma_2$. Suppose that a function $u$ be sub-elliptic in $\Omega\cap B(0,aR)$, $u>0$ in $\Omega$, 
$u=0$ on $\Gamma_1\cap B(0,aR)$ and $\frac{\partial u}{\partial \ell}\leq 0$ on $\Gamma_2\cap B(0,aR)$. Let $\Gamma_2$ satisfy inner cone condition.
 
Assume that there is a barrier $w$ in balls $B(0,R)$ and $B(0,aR)$. 
 
 Then
  \begin{equation}\label{boundlayer_strong}
\sup_{\Omega\cap B(0,a R)} u\ge \frac{\sup_{\Omega\cap B(0,R)} u}{1-\eta_0}.
\end{equation}  
\end{lemma}

\begin{proof}
Let $M=\sup_{\Omega\cap B(0,aR)} u$, and let the barrier $w(x)$ be as in Definition \ref{strict_growth}. Define
\begin{equation*}
v(x)=M(1-w(x)).
\end{equation*} 
Obviously ${\cal L} v \geq {\cal L} u$ in $\Omega$,  $v\geq u$ on $\Gamma_1\cap B(0,aR)$,  $ \frac{\partial v}{\partial \ell}\ge \frac{\partial u}{\partial \ell}$ on 
$\Gamma_2$, and $v\geq M \geq u$ on $\partial B(0,aR)\cap \Omega.$ Applying comparison Lemma \ref{comparison_theorem} to functions $v$ and $u$ in the domain 
$\Omega\cap B(0,aR)$ we get that $v\geq u$. In the intersection $\Omega\cap B(0,R)$ this gives with regard of (\ref{w>eta})
\begin{equation*}
M(1-\eta_0)\geq M(1-\inf_{\Omega\cap B(0,R)} w)\geq  \sup_{\Omega\cap B(0,R)} u. 
\end{equation*}
The latter is equivalent to statement in \eqref{boundlayer_strong}.
\end{proof}

We  recall the well-known notion of $s$-capacity, see, e.g., \cite[Sec. I.2]{landis-book}. 
\begin{definition}\label{s-capacity}
Let $H$ be a Borel set. Let a measure $\mu$ be defined on Borel subsets of $H$. We call $\mu$ {\bf admissible} and write $\mu \in {\cal M}(H)$ if 
\begin{equation*}
\int_{H} \frac{d\mu(y)}{|x-y|^s} \leq 1, \quad  \text{for} \quad x\in R^n\setminus H. 
\end{equation*}  
Then the quantity
\begin{equation*}
{\bf C}_s(H)=\sup_{\mu\in{\cal M}(H)} \mu(H) 
\end{equation*}
is called {\bf $s$-capacity} of $H$.
\end{definition}

We also recall the following simple statement.
 \begin{proposition}\label{subsol}
If $s\ge e_1-2 $ then $L|x|^{-s}\le 0$. 
\end{proposition}

Now we formulate a variant of the Landis Growth Lemma, see \cite[Sec. I.4]{landis-book}.

\begin{lemma}\label{landis}
 Let function $u$ be sub-elliptic in $\Omega\cap B(0,aR)$, $u>0$ in $\Omega$, $u =0$ on $\Gamma_1=\partial \Omega \cap B(0,aR)$. 
 Let $s\ge e_1-2$. Then there exists $0<\eta_1<1$ depending only on $s$ s.t.
 \begin{equation*}\label{boundlayer_capacity}
\sup_{\Omega\cap B(0,a R)} u\ge \frac{\sup_{\Omega\cap B(0,R)} u}{1-\eta_1 {\bf C}_s(H)R^{-s}}.
\end{equation*} 
Here $H=\Gamma_1\cap B(0,R)$.

Consequently if $B(0,R)\setminus \Omega$ contains a ball with radius $\delta R$ then  
\begin{equation*}\label{boundlayer_constant}
\sup_{\Omega\cap B(0,a R)} u\ge \frac{\sup_{\Omega\cap B(0,R)} u}{1-\widetilde\eta_1},
\end{equation*} 
where the constant $\widetilde\eta_1$ depend on $s$ and $\delta$.
\end{lemma}

\section{Growth Lemma near Neumann boundary}\label{sec:Lip-boundary}

%

Here we prove the Growth Lemmas in the domain adjunct to $\Gamma_2$ under some assumption on $\Gamma_1$. 

We recall that $\Gamma_2$ is uniformly Lipschitz in a neighborhood of $x^0$. This means that there is $\delta>0$ s.t. the set $\Gamma_2\cap B(x^0,\delta)$ 
is the graph $x_n=f(x')$ in a local Cartesian coordinate system, and the function $f$ is Lipschitz. Moreover, we suppose that its Lipschitz constant does not 
exceed $L$. Without loss of generality we assume that $\Omega\cap B(x^0,\delta)\subset\{x_n<f(x')\}$ (see Fig. \ref{Lipschitz}). This implies the inner cone condition
if we direct the axis of the cone $K$ along $-x_n$ and set $\varphi=\cot^{-1}(L)$.
\begin{figure}[ht]
\begin{center}
\advance\leftskip-3cm
\advance\rightskip-3cm
\includegraphics[scale=0.5]{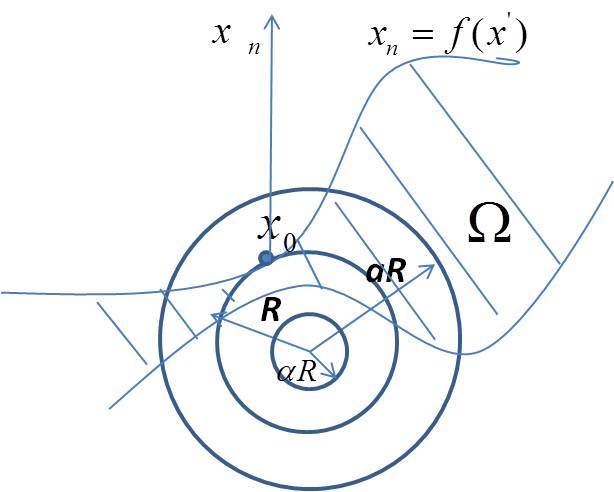}
\caption{Domain $\Omega$, boundary $\Gamma_2$ and balls $B(0,R)$, $B(0,aR)$ and $B(0,\alpha R)$.}
\label{Lipschitz}
\end{center}\end{figure}

\begin{lemma}\label{bound_layer}
Let $\Gamma_2\cap B(0,R)=\emptyset$, and $x^0\in\Gamma_2\cap\partial B(0,R)$, for some $R\le\frac\delta 2$. Assume that
$\Omega\cap B(0,\alpha R)=\emptyset$ for some $0<\alpha<\frac 12$ (see Fig. \ref{Lipschitz}).

Suppose that the vector field $\ell$ satisfies conditions in Lemma \ref{nad} uniformly on $\Gamma_2$ (that is, $\varepsilon$ 
does not depend on $x\in\Gamma_2$). 

Let function $u$ be sub-elliptic (${\cal L}u\le 0$ in $\Omega$), $u>0$ in $\Omega$, $u=0$ on $\Gamma_1$ and $\frac{\partial u}{\partial \ell}\le 0$ on $\Gamma_2$.

Then there exists $a>1$ depending on the Lipschitz constant $L$, $\varepsilon$ and ellipticity constant $e_1$ 
s.t.  
\begin{equation}\label{boundlayer}
\sup_{\Omega\cap B(0,a R)} u\ge \frac{\sup_{\Omega\cap B(0,R)} u}{1-\eta_2}.
\end{equation}
Here $\eta_2\in (0,1)$ is defined by $\alpha$ and $a$.
\end{lemma}

\begin{proof}
We take $s\ge e_1-2$ and set
\begin{equation*}\label{w} 
w(x)=\frac {\alpha^s R^s}{|x|^s}-\frac {\alpha^s}{a^s}. 
\end{equation*} 
We claim that for $a$ sufficiently close to $1$  this function satisfies all conditions in Definition \ref{strict_growth}.
Indeed:

1. From Proposition \ref{subsol}  function $w$ is sub-elliptic, condition \eqref{Lw<0} holds.

2. Evidently $w=0$ on $\partial B(0,aR)$, condition \eqref{w<0} holds,

3. while $\Omega\cap B(0,\alpha R)=\emptyset$ implies $ w\le 1$ in $\Omega\cap B(0,aR)$ (and therefore on $\Gamma_1$) condition \eqref{w<1} holds. 

Now we check condition \eqref{w_l<0}. We introduce the Cartesian coordinate system with
axes collinear with those of local coordinate system at $x^0$. We observe that the assumption $\Gamma_2\cap B(0,R)=\emptyset$ and Lipschitz condition
imply that for $x\in \Gamma_2\cap B(0,aR)$
\begin{equation*}\label{mod-x'-x_n-prp}
|x'|\le \frac R{\sqrt{1+L^2}}\,(L+\sqrt{a^2-1});\qquad x_n\ge \frac R{\sqrt{1+L^2}}\,(1-L\sqrt{a^2-1}).
\end{equation*}

Moreover, our assumption on the vector field $\ell$ means that
\begin{equation*}\label{l'-prop}
|\ell'|\le \sin(\cot^{-1}(L)-\varepsilon)\le \frac 1{\sqrt{1+L^2}}-\widetilde\varepsilon;
\end{equation*}
\begin{equation*}\label{l_n-prop}
 \ell_n\geq \cos(\cot^{-1}(L)-\varepsilon)\geq
\frac L{\sqrt{1+L^2}}+\widetilde\varepsilon
\end{equation*}
where $\widetilde\varepsilon$ depends only on $L$ and $\varepsilon$.

Therefore, the direct calculation gives
 \begin{multline*}
\frac{\partial w}{\partial \ell}(x)=-\,\frac{s\alpha^sR^s}{|x|^{s+2}}\cdot(x_n \ell_n+\ell'\cdot x')\\
\le\frac{s\alpha^sR^s}{|x|^{s+2}}\cdot\frac R{\sqrt{1+L^2}} \Big(\sqrt{a^2-1}\cdot\big(\sqrt{1+L^2}+\widetilde{\varepsilon}(L-1)\big)
-\widetilde{\varepsilon}\big(L+1\big)\Big).
 \end{multline*}   
It is easy to see that, given $\widetilde\varepsilon>0$, there is $a>1$ depending only on $\widetilde\varepsilon$ and $L$ s.t. 
$\frac{\partial w}{\partial \ell}(x)\le 0$, and \eqref{w_l<0} holds.

Finally, for $x\in \Omega\cap B(0,R)$, $w(x)\ge \alpha^s(1-a^{-s})=:\eta_2$, and \eqref{w>eta} holds. 

Thus, the claim follows, and $w$ is the barrier in the balls  $B(0,R)$, $B(0,aR)$.  From Lemma \ref{krylov} we get \eqref{boundlayer}.
\end{proof}

\section{Growth Lemma in the Spherical Layer}\label{sec:Spherical-Layer}

In this section we prove Growth Lemma in spherical layer near junction  point of interest  $\zeta = \overline\Gamma_1 \cap \overline\Gamma_2$.  
Without loss of generality we put $\zeta = 0 $. 

First we will introduce admissible class of domains in the spherical layer.

\begin{definition}\label{admissible-domain}
Fix five constants $0<q_1<q_2<q^*<q_3<q_4$. Define two spherical layers $\hat{U}_R\subset U_R$:
$$
U_R=B(0,q_4R)\setminus B(0,q_1R); \qquad \hat{U}_R=B(0,q_3 R) \setminus B(0,q_2 R).
$$

We call $\Omega$ {\bf admissible} in the layer $U_R$ if for some $\theta>0$ there is finite set of the balls (see Fig. \ref{Admissible})
\begin{equation*}\label{B_k} 
{\cal B}=\{B^k=B(\xi_k,\theta R)\}_{k=0}^{N}; \quad   B^k\subset\hat{U}_R  
\end{equation*} 
s.t. the following holds:


1. ${\bf C}_s(B^0\cap \Gamma_1)\geq \varkappa {\bf C}_s(\Gamma_1\cap \hat{U}_R)$, for some constant $\varkappa>0$.

2. $B^k\cap \Gamma_2 = \emptyset$, $k=1,..,N$, and $B(\xi_0,a\theta R)\cap \Gamma_2 = \emptyset$, where $a>1$ is defined in Lemma \ref{bound_layer}. 
\footnote{Note that boundaries of some balls $B^k$ may touch $\Gamma_2$. }


3. There is $\delta\in (0,1/2)$ s.t. every ball in $\cal B$ can be connected with $B^0$ by a subsequence of balls $B^j$ s.t. any intersection $ B^j\cap B^{j+1}\cap \Omega$ 
contains the ball $B(\xi_{j+1},\delta R)$.


4. The set $S_R=\partial B(0,q^* R)\cap \Omega$ is covered by balls in $\cal B$. 



\end{definition}

Fig. \ref{Admissible} schematically illustrate Definition \ref{admissible-domain}.

\begin{figure}
\centering
\begin{subfigure}{.5\textwidth}
  \centering
  \includegraphics[width=1\linewidth]{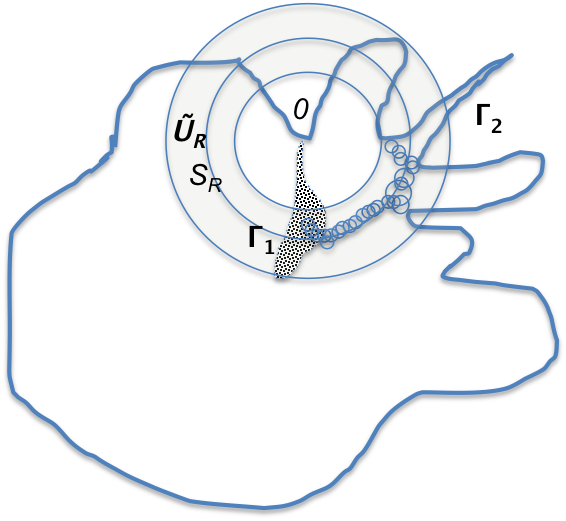}
  \label{fig:sub1}
\end{subfigure}%
\begin{subfigure}{.5\textwidth}
  \centering
 \advance\leftskip-1cm
 \includegraphics[width=2\linewidth]{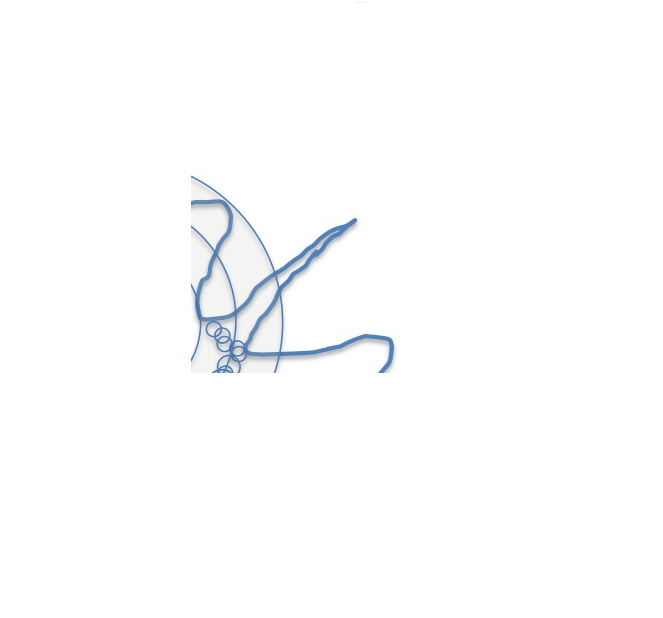}
  \label{fig:sub2}
\end{subfigure}
\vspace{-2 cm}
\caption{On the left: domain $\Omega$ admissible in Spherical Layer $U_R$.  On the right: domain and layer zoomed near boundary $\Gamma_2$ (bold line).}
\label{Admissible}

\end{figure}

\begin{lemma}\label{Lipschitz_layer}
Let function $u$ be sub-elliptic,  $u>0$ in $\Omega$. Suppose that $u\leq 0$ on $\Gamma_1$ and $\frac{\partial u}{\partial \ell}\leq 0$ on $\Gamma_2$. 
Let domain $\Omega$ be admissible in the layer $U_R$.
Then
\begin{equation*}\label{lemmainlayer}
\sup_{\Omega}u\geq \frac{\sup_{S_R} u}{1-\eta {\bf C}_s(H)R^{-s}}.
\end{equation*}
Here $H=\Gamma_1\cap \hat{U}_R$ while $\eta$ depends on $s$, the ellipticity constant $e_1$, the Lipschitz constant $L$, the vector field $\ell$, constants $\theta$, 
$\varkappa$, $\delta$ in Definition \ref{admissible-domain} and the number $N$ of balls in the set $\cal B$.
\end{lemma}

\begin{proof}
Without loss of generality we set $\theta=1$. Let $\sup_{S_R} u=:m=u(y)$, here $y\in \overline S_R$. By assumption 4 in Definition \ref{admissible-domain},
$y\in {\overline B}\vphantom{B}^k$ for some $k$. By assumption 3, we can choose a subsequence $B^j$ connecting $B^0$ and $B^k$.

Consider the ball $B_0$ and the ball  $B(\xi_0,aR)$, $a>1$, concentric to it. Due to assumptions  1 and 2 in Definition \ref{admissible-domain}, 
we can apply  Lemma \ref{landis} to get:  
\begin{equation*}\label{B1}
M:=\sup_{\Omega} u \geq \sup_{\Omega\cap B(\xi_0,a R)}u \geq \frac{\sup_{B^0\cap \Omega} u}{1-\varkappa\eta_1 {\bf C}_s(H)R^{-s}}. 
\end{equation*}  

Suppose that
\begin{equation} \label{delta0}
 \sup_{B^0\cap \Omega} u \geq m(1-\delta_0), \qquad \text{where} \quad \delta_0 = \frac{\varkappa\eta_1 {\bf C}_s(H)R^{-s}}{2(1-\varkappa\eta_1 {\bf C}_s(H)R^{-s})}. 
\end{equation}
Then after some calculation we get 
\begin{equation*}\label{B11}
M \geq  \frac{m}{1-\eta_3 {\bf C}_s(H)R^{-s}}
\end{equation*} 
for some $\eta_3$ depending on $\varkappa\eta_1$, and the statement follows. 

If \eqref{delta0} does not hold, we consider the function 
\begin{equation}\label{func_v1}
u_1(x)=u(x)-m(1-\delta_0),
\end{equation}
then $u_1(x)\leq 0$ in $B^0 \cap \Omega$. 

By assumption 3,  $B^0 \cap B^1 \cap \Omega$ contains a ball of radius $\delta R$. 
Let $\Omega_1:=\{ x: u_1 (x)>0\}$.  Assume that $B^1\cap \Omega_1\not= \emptyset$, otherwise we consider the first ball in the subsequence $B^j$ for which 
this property holds.

Suppose that 
\begin{equation}\label{tau}
 \sup_{B^1\cap \Omega} u_1 \geq m\delta_0(1-\tau), 
\end{equation}
here the constant $\tau$ will be chosen later.

Consider any simply connected component of the domain $B(\xi_1,aR)\cap \Omega_1$ in which the supremum in (\ref{tau}) is realised.
There are two possibilities:

a) $B(\xi_1,aR)\cap \Gamma_2 = \emptyset$; 

b) $B(\xi_1,aR)\cap \Gamma_2 \ne \emptyset$ 

\noindent (recall that $a=a(L,\ell, e_1)>1$ is defined in Lemma 3.1).

Let us start with case (a). Due to assumption 3, Lemma \ref{landis} and (\ref{tau}) it follows that
\begin{equation}\label{func_v11}
\sup_{B(\xi_1,aR)\cap \Omega} u_1\ge  \frac {\sup_{B^1\cap \Omega} u_1}{1-\widetilde\eta_1}\ge \frac{m\delta_0(1-\tau)}{1-\widetilde\eta_1}.
\end{equation}

Using  (\ref{func_v1}) and \eqref{func_v11} we deduce
\begin{equation*}\label{u-tau1}
 \sup_{B(\xi_1,aR)\cap \Omega}u \geq m\big(1+\frac {\delta_0(\widetilde\eta_1-\tau)}{1-\widetilde\eta_1}\big).  
\end{equation*}
Letting $\tau=\frac{\widetilde\eta_1}2$ we get 
\begin{equation}\label{u-eta4}
M \geq  \sup_{B(\xi_1,aR)\cap \Omega}u \geq m\big(1+\frac {\delta_0\tau}{1-2\tau}\big),  
\end{equation}
and the statement follows.

In case of (b) we proceed with the same arguments but instead of Lemma  \ref{landis} we apply Lemma \ref{bound_layer} and put $\tau=\frac{\eta_2}2$.
Thus, if (\ref{tau}) holds with $\tau=\frac 12\min\{\widetilde\eta_1,\eta_2\}$ then (\ref{u-eta4}) is satisfied in any case, and Lemma is proved.

If \eqref{tau} does not hold then function $u$ satisfies
\begin{equation*}\label{u-tau}
 \sup_{B^1\cap \Omega}u \leq m(1-\delta_0\tau). 
\end{equation*}

As in previous step we consider the function
\begin{equation*}\label{func_v2}
u_2(x)=u(x)- m(1-\delta_0\tau),
\end{equation*}
$u_2(x)\leq 0$ in $B^1 \cap \Omega$.

Repeating previous argument we deduce that if
\begin{equation} \label{v-2}
\sup_{B^2\cap \Omega} u_2 \ge m\delta_0 \tau(1 -\tau)
\end{equation}
then
$$
M \geq m\big(1+\frac {\delta_0\tau^2}{1-2\tau}\big),  
$$
and Lemma is proved.

If \eqref{v-2} does not hold, then 
$$
 \sup_{B^2\cap \Omega}u \leq m(1-\delta_0\tau^2). 
$$
Repeating this process we either prove Lemma or arrive at the inequality
$$
 \sup_{B^k\cap \Omega}u \leq m(1-\delta_0\tau^k) 
$$
that is impossible since $y\in {\overline B}\vphantom{B}^k$ and $u(y)=m$.
\end{proof}

\section{Dichotomy of solutions}\label{sec:Dichotomy}

In this section we will apply obtained Growth Lemma in spherical layer to prove dichotomy of solutions near point $\zeta$  of the junction of Dirichlet and Neumann 
boundaries. As in previous section we put $\zeta=0$.

Let $\Omega \subset \{x :x_n <f(x')\}$ and $\Gamma_2$ is a graph of the function $x_n=f(x')$, $f(0)=0$.  Set $R_m=Q^{-m}$ for some $Q>1$, $S_m=\partial B(0, q^*R_m)$, and
$$
U_m=B(0, q_4 R_m) \setminus B(0, q_1 R_m), \quad \hat{U}_m=B(0,q_3R_m)\setminus B(0, q_2 R_m).
$$

We fix $N_0\in\mathbb N$ and $q_1<q_2<q^*<q_3<q_4$ s.t. $q^*<q_1Q$. Suppose that for all $m\ge N_0$ the domain $\Omega$ with boundaries $\Gamma_1$ and $\Gamma_2$ is 
admissible in the layer $U_m$ in the sense of Definition \ref{admissible-domain} with $R=R_m$, and all constants in Definition \ref{admissible-domain} do not depend on $m$.


\begin{lemma}\label{Dich_main}
Let function $u$ be sub-elliptic,  $u>0$ in $\Omega$. Suppose that $u\leq 0$ on $\Gamma_1\cap B(0,q_4R_{N_0})$ and $\frac{\partial u}{\partial \ell}\leq 0$ 
on $\Gamma_2\cap B(0,q_4R_{N_0})$. Let domain $\Omega$ be admissible in the layers $U_m$, $m\ge {N_0}$.

Let $M_m=\sup_{S_m\cap\Omega} u$. Then one of two statements holds: 
\medskip

either $M_{N_1+1}\ge M_{N_1}$ for some $N_1$, and for all $m> N_1$
\begin{equation}\label{m+1>m}
M_{m+1}\geq \frac{M_m}{1-\eta {\bf C}_s(H_m)Q^{sm}},
\end{equation}

or for all $m>{N_0}$
\begin{equation}\label{m>m+1}
M_m\geq \frac{M_{m+1}}{1-\eta {\bf C}_s(H_m)Q^{sm}}.
\end{equation}
Here $H_m=\Gamma_1\cap \hat{U}_m$, and $\eta$ is the constant from Lemma \ref{Lipschitz_layer}.
\end{lemma}

\begin{proof}
Due to Lemma \ref{comparison_theorem}, there are two possibilities:
 
(a) if $M_{N_1+1}\ge M_{N_1}$for some $N_1>{N_0}$ then $M(\rho)=\sup_{\partial B(0, \rho)\cap \Omega} u>M_m$, $m>N_1$ for any $\rho < q^* R_m$;

(b)  otherwise $M_m>M_{m+1}$ for all $m>{N_0}$. 

Now Lemma \ref{Lipschitz_layer} gives (\ref{m+1>m}) in the case (a) and (\ref{m>m+1}) in the case (b).
\end{proof}

\begin{remark}\label{mon}
Let function $u$ be sub-elliptic,  $u>0$ in $\Omega$. Suppose that $u\leq 0$ on $\Gamma_1\cap B(0,\rho_0)$ and $\frac{\partial u}{\partial \ell}\leq 0$ 
on $\Gamma_2\cap B(0,\rho_0)$. Then the maximum principle implies the following dichotomy (we recall that $M(\rho)=\sup_{\partial B(0, \rho)\cap \Omega} u$):\medskip

either there is $\rho^*\le\rho_0$ s.t.   for $\rho_2<\rho_1<\rho^*$ we have $M(\rho_2)>M(\rho_1)$;

or $M(\rho_2)<M(\rho_1)$ for all $\rho_2<\rho_1<\rho_0$.
\end{remark}

Applying recursively alternative in Lemma \ref{Dich_main} and using Remark \ref{mon} we get asymptotic dichotomy.

\begin{theorem}
 Let the assumptions of Lemma \ref{Dich_main} be satisfied.
Suppose that $\sum_{m=0}^{\infty} {\bf C}_s(H_m)Q^{sm}=\infty$, where $H_m=\Gamma_1\cap \hat{U}_m$.

Then one of two statements holds: 
\medskip

either $M(\rho)\to\infty$ as $\rho \to 0$, and
\begin{equation*}
\liminf_{\rho\to\infty} M(\rho)\exp\Big(-\widehat\eta\sum_{m=0}^{[c\ln \rho ]} {\bf C}_s(H_m)Q^{sm}\Big)>0 ,
\end{equation*}

or $M(\rho)\to0$ as $\rho \to 0$, and
\begin{equation*}
\limsup_{\rho\to\infty} M(\rho)\exp\Big(\widehat\eta\sum_{m=0}^{[c\ln \rho ]} {\bf C}_s(H_m)Q^{sm}\Big)=0 ,
\end{equation*}

Here $\widehat\eta$ and $c$ depend on the same quantities as $\eta$ in Lemma \ref{Lipschitz_layer}.
\end{theorem}

\bibliographystyle{plain}

\end{document}